\documentclass[12pt]{amsart}
\usepackage{enumerate}
\usepackage{amsmath, amssymb, amsthm, amsfonts}
\usepackage{tikz}
\usepackage{geometry}
\usepackage{bbm} %for the 1 symbol of indicator functions
\usepackage{url}

\newtheorem{theorem}{Theorem}[section]
\newtheorem{proposition}[theorem]{Proposition}
\newtheorem{corollary}[theorem]{Corollary}
\newtheorem{lemma}[theorem]{Lemma}

\newtheorem{claim}[theorem]{Claim}

\theoremstyle{definition}
\newtheorem{remark}[theorem]{Remark}

\newtheorem{definition}[theorem]{Definition}

\newtheorem{example}[theorem]{Example}

\def\E{\mathbb{E}}

\def\IR{\mathbb{R}}
\def\IZ{\mathbb{Z}}
\def\IN{\mathbb{N}}

\def\IB{\mathcal{B}}
\def\II{\mathcal{I}}

\def\al{\alpha}

\def\ga{\gamma}
\def\Ga{\Gamma}

\def\si{\sigma}

\def\Tt{\tilde{T}_{d-1}}
\def\Gat{\tilde{\Gamma}}

\DeclareMathOperator{\Aut}{Aut}

\DeclareMathOperator{\var}{var}
\DeclareMathOperator{\cov}{cov}
\DeclareMathOperator{\corr}{corr}

\DeclareMathOperator{\dist}{dist}

\newcommand{\defeq}{\mathrel{\vcenter{\baselineskip0.5ex \lineskiplimit0pt
                     \hbox{\scriptsize.}\hbox{\scriptsize.}}}%
                     =}

\def\ind{\mathbbm{1}} %the 1 symbol of indicator functions

\begin{document}

\title[Correlation bound for distant parts of FIID processes]
{Correlation bound for distant parts\\of factor of IID processes}

\author[Backhausz]{\'{A}gnes Backhausz}
\address{E\"otv\"os Lor\'and University, Department of Probability and Statistics 
H-1117 Budapest, P\'azm\'any P\'eter s\'et\'any 1/c; 
and MTA Alfr\'ed R\'enyi Institute of Mathematics 
H-1053 Budapest, Re\'altanoda utca 13-15}
\email{agnes@cs.elte.hu}

\author[Gerencs\'{e}r]{Bal\'{a}zs Gerencs\'{e}r}
\address{MTA Alfr\'ed R\'enyi Institute of Mathematics 
H-1053 Budapest, Re\'altanoda utca 13-15;
and E\"otv\"os Lor\'and University, Department of Probability and Statistics 
H-1117 Budapest, P\'azm\'any P\'eter s\'et\'any 1/c}
\email{gerencser.balazs@renyi.mta.hu}

\author[Harangi]{Viktor Harangi}
\address{MTA Alfr\'ed R\'enyi Institute of Mathematics 
H-1053 Budapest, Re\'altanoda utca 13-15}
\email{harangi@renyi.hu}

\author[Vizer]{M\'{a}t\'{e} Vizer}
\address{MTA Alfr\'ed R\'enyi Institute of Mathematics 
H-1053 Budapest, Re\'altanoda utca 13-15}
\email{vizermate@gmail.com}

\thanks{
The first author was partially supported by the Hungarian Scientific Research 
Fund (OTKA, grant no. K109684), and the MTA R\'enyi Institute ``Lend\"ulet'' Limits 
of Structures Research Group. 
The third author was supported by Marie Sk{\l}odowska-Curie Individual Fellowship Grant No.\ 661025 
and the MTA R\'enyi Institute ``Lend\"ulet'' Groups and Graphs Research Group. 
The fourth author was supported by Hungarian National Scientific Fund, grant number: SNN-116095, 
and was partially supported by ERC Consolidator Grant 648017.}



\date{}

\begin{abstract}
We study factor of i.i.d.\ processes on the $d$-regular tree for $d \geq 3$. 
We show that if such a process is restricted to two distant connected subgraphs of the tree, 
then the two parts are basically uncorrelated. More precisely, 
any functions of the two parts have correlation at most 
$k(d-1) / (\sqrt{d-1})^k$, where $k$ denotes the distance of the subgraphs.
This result can be considered as a quantitative version of the fact 
that factor of i.i.d.\ processes have trivial 1-ended tails. 
\end{abstract}

\maketitle

%%%%%%%%%%%%%%%%%%%%%%%%%%%%%%%%%%%%%%%%%%%%%%%%%%%%%%%%%%%%%%%%%%%%%%%%%%%%%%%%%%%%%%%%%%%%%
%%                                                                                         %%
%%        INTRODUCTION                                                                     %%
%%                                                                                         %%
%%%%%%%%%%%%%%%%%%%%%%%%%%%%%%%%%%%%%%%%%%%%%%%%%%%%%%%%%%%%%%%%%%%%%%%%%%%%%%%%%%%%%%%%%%%%%

\section{Introduction}

%\subsection{Results}

This paper deals with factor of i.i.d.\ processes
on the $d$-regular tree $T_d$ for $d\geq 3$.
Loosely speaking, we first put independent and identically distributed
(say $[0,1]$ uniform) random labels on the vertices of $T_d$;
then each vertex gets a new label that depends on
the labelled rooted graph as seen from that vertex,
all vertices ``using the same rule''.

For a formal definition, let $V(T_d)$ denote the vertex set and
$\Aut(T_d)$ the automorphism group of $T_d$. Suppose that $M$ is a measurable space
and $F \colon [0,1]^{V(T_d)} \to M^{V(T_d)}$ is a measurable function.
Then $F$ is said to be an $\Aut(T_d)$-factor (or factor in short)
if it is $\Aut(T_d)$-equivariant, that is,
it commutes with the natural $\Aut(T_d)$-actions.
(In most applications $M$ is either a discrete set or $\IR$.)
%Here the rule is a function $f \colon [0,1]^{V(T_d)} \to M$ that ...
Let $\pi \colon M^{V(T_d)} \to M$ denote
the coordinate projection corresponding to a distinguished vertex.
Then the function $f = \pi \circ F \colon [0,1]^{V(T_d)} \to M$
(often called the \emph{rule})
will be invariant under the stabilizer of the distinguished vertex.
It is easy to see that $F$ is determined by $f$.

If we have an i.i.d.\ process $Z = \left( Z_v \right)_{v \in V(T_d)}$
on $[0,1]^{V(T_d)}$, then applying $F$ yields a factor of i.i.d.\ process $X = F(Z)$,
which can be viewed as a collection $X = \left( X_v \right)_{v \in V(T_d)}$ of
$M$-valued random variables. It follows immediately from the definition
that the distribution of $X$ is invariant under the action of $\Aut(T_d)$;
in particular, each $X_v$ has the same distribution.

A natural question is ``how independent'' the random variables $X_v$ are.
It is fairly easy to see that for any factor
the correlation of $X_u$ and $X_v$ converges to $0$ as
the distance of $u$ and $v$ goes to infinity.
In \cite{cordec} the following sharp bound was found for the correlation:
\begin{equation} \label{eq:corr_decay_for_vertices}
\left| \corr( X_u, X_v ) \right| \leq
\left( k+1 - \frac{2k}{d} \right) \left( \frac{1}{ \sqrt{d-1} } \right)^k
\mbox{, where } k = \dist(u,v) ,
\end{equation}
that is, the rate of the correlation decay is essentially $1 / (\sqrt{d-1})^k$.
(Here it is assumed that $M=\IR$ and $\var X_v < \infty$.)

\subsection{Results}
The main result of this paper basically says that
if two connected subsets $V_1, V_2 \subset V(T_d)$ have large distance,
then they are ``almost independent'' in the following sense:
for an arbitrary factor $X$, any function of $\left( X_v \right)_{v \in V_1}$
and any function of $\left( X_v \right)_{v \in V_2}$ have small correlation,
essentially of (the optimal) order $1 / (\sqrt{d-1})^k$.
\begin{theorem} \label{thm:main}
Let $X = \left( X_v \right)_{v \in V(T_d)}$ be a factor of i.i.d.\ process
on $M^{V(T_d)}$ for some measurable space $M$,
and let $V_1, V_2 \subset V(T_d)$ be arbitrary (possibly infinite) subsets of the vertex set.
Suppose that $h_i \colon M^{V_i} \to \IR$ is a measurable function, $i=1,2$.
For $h_i\left( \left( X_v \right)_{v \in V_i} \right)$ we simply write $h_i(X)$.
If $h_1(X)$ and $h_2(X)$ have finite variances, then we have
\begin{equation} \label{eq:corr_of_subgraphs}
\left| \corr( h_1(X) , h_2(X) ) \right| \leq
k(d-1) \left( \frac{ 1 }{ \sqrt{d-1} } \right)^k,
\end{equation}
where $k$ denotes the distance of the convex hulls of $V_1$ and $V_2$,
which we assume to be positive.
(The convex hull of $V_i$ is the smallest
connected subgraph of $T_d$ containing $V_i$.)
\end{theorem}
One might wonder if a similar bound could exist
if $k$ denoted the distance of $V_1$ and $V_2$
instead of the distance of the convex hulls.
The simplest case where this would make a difference is
when $V_1$ consists of the two endpoints of a path of length $2k$
and $V_2$ is the one-element set containing the midpoint of this path.
Then the distance of $V_1$ and $V_2$ is $k$, while the distance of the convex hulls is $0$.
In this case the above theorem would be of no use.
Can we still have a good bound for the correlation?
The answer is negative, as the correlation might actually be $1$ in this case for any $k$.
This will be shown by Example \ref{ex},
where we will construct a factor of i.i.d.\ process $X$ on $[0,1]^{V(T_d)}$
with the property that $X_u$ and $X_v$ determine
the values of $X$ along the whole path connecting $u$ and $v$.
In fact, this process will show that 
\eqref{eq:corr_of_subgraphs} is essentially sharp: 
for any $V_1, V_2$ there exist $h_1, h_2$
such that the correlation in question is of order $1 / (\sqrt{d-1})^k$
where $k$ is the distance of the convex hulls. 

Given an invariant process $\left( X_v \right)_{v \in V(T_d)}$,
an event is in the \emph{tail} of the process if it is, for each $r \in \IN$,
contained by the $\si$-algebra generated by the random variables $X_v$
for vertices $v$ outside the $r$-ball around a fixed root.
It is open whether any invariant process with a trivial tail 
can be obtained as some factor of an i.i.d.\ process. 
As for the other direction, there are examples for factor of i.i.d.\ processes
whose tail is not trivial. In fact, they can even have full tails.
(See \cite[Proposition 2.4]{lyons} or Example \ref{ex} of the current paper.)
Our result is related to another kind of tails called the \emph{1-ended tails}: 
given an infinite path starting at the root,
consider those events that are, for each $r \in \IN$, contained by
the $\si$-algebra generated by $X_v$'s
for vertices $v$ of $T_d$ that are separated
from the root by the $r$th vertex of the path. 

Let us consider Theorem \ref{thm:main} in the special case 
when $h_1, h_2$ are indicator functions of two events 
with one event being fixed and the other running through a sequence of events 
in such a way that the distance $k$ goes to infinity. 
Then we obtain that the sequence is ``asymptotically independent'' from the fixed event. 
This is actually equivalent to the triviality of the 1-ended tails, see Remark \ref{rm:eq}. 
Therefore the following is an immediate consequence of Theorem \ref{thm:main}.
\begin{corollary} \label{cor:1-ended}
The 1-ended tail $\sigma$-algebras are trivial 
for any factor of i.i.d.\ process on $M^{V(T_d)}$ for $d \geq 3$.
\end{corollary}
As pointed out by Russell Lyons \cite[Section 2]{lyons}, 
this corollary was known more generally: Pemantle showed that 
basically every ergodic invariant process on $T_d$ 
has trivial 1-ended tails \cite[Theorem 1]{pemantle}. 
See \cite[Corollary 7]{pemantle} for an equivalent formulation 
(using asymptotic independence as described above). 
Our main theorem provides a quantitative version of this formulation 
in the case of factor of i.i.d.\ processes: 
we obtain a universal bound for an arbitary factor and 
for arbitrary events (only depending on the distance $k$). 
It is also worth mentioning that the triviality of 1-ended tails implies mixing and 
a weak law of large numbers \cite[Corollary 8-10]{pemantle}. 

Corollary \ref{cor:1-ended} will be complemented by Example \ref{ex}, 
where we construct a factor of i.i.d.\ process for which, 
loosely speaking, ``any tail broader than a 1-ended tail'' is non-trivial. 
This means that Corollary \ref{cor:1-ended} is, in some sense, the best we can hope for. 

Note that Corollary \ref{cor:1-ended} is also true for $d=2$ (that is, on $\IZ$) 
by a result of Rokhlin and Sinai \cite{rokhlin1961construction}. 

We finish this section with a brief outline of the proof of Theorem \ref{thm:main}.
Let $\Tt$ denote the rooted $(d-1)$-ary tree.
%Note that if we delete an edge of $T_d$, we get two copies of $\Tt$.
It is easy to see that we might assume that
$V_1$ and $V_2$ are both isomorphic to $\Tt$
and their roots have distance $k$. Let $e_i$ be the directed edge 
starting at the root of $V_i$ and ``pointing away'' from $V_i$.
We will show that $\left| \corr( h_1(X) , h_2(X) ) \right|$ is maximized
by functions $h_1$ and $h_2$ that are invariant
under the automorphism group of $\Tt$.
In fact, they should ``come from'' the same
$\Aut(\Tt)$-invariant measurable function $f \colon M^{V(\Tt)} \to \IR$.
Given such a function $f$,
an $\Aut(T_d)$-invariant process on the vertices of $T_d$ can be turned into
an $\Aut(T_d)$-invariant process on the directed edges of $T_d$:
for any directed edge $e$ of $T_d$,
apply $f$ to the (labelled) subgraph ``behind'' $e$ and write its value on $e$.
The process we obtain on the directed edge set $E(T_d)$
will be a factor of i.i.d.\ process (see Section \ref{sec:2} for precise definitions)
whose value on $e_i$ is $h_i(X)$, $i=1,2$.
Therefore to complete the proof we need to prove a correlation decay result
similar to \eqref{eq:corr_decay_for_vertices}
but for directed edges instead of vertices (see Theorem \ref{thm:edge_corr}).
To this end we will need to estimate the norms of the powers
of the non-backtracking operator (see Section \ref{sec:4.1}).

\subsection{Related work}
Factor of i.i.d.\ processes can be viewed from an ergodic theoretic point of view,
namely, as factors of the Bernoulli shift. $\IZ$-factors (as part of
classical ergodic theory) have the largest literature and the most complete theory.
For amenable group actions entropy serves as a complete invariant (for isomorphism of
i.i.d.\ processes). A classical example of Ornstein and Weiss \cite{ornstein_weiss_1987}
expresses the $4$-shift as a factor of the $2$-shift over free groups of rank at least $2$, 
showing that no notion of entropy can exist in the non-amenable case
that would exhibit all the nice properties of the classical Kolmogorov-Sinai entropy.
Nevertheless, various definitions of entropy have been introduced
and examined in relation to factor maps, see e.g.\ \cite{bowen, sofic, kerr, brandon}.

One of the reasons why factor of i.i.d.\ processes have attracted a growing attention
in recent years is that they give rise to some sort of randomized local algorithms
that can be carried out on arbitrary regular graphs with ``large essential girth'',
e.g.\ random regular graphs. Such factor of i.i.d.\ constructions include
perfect matchings \cite{csokalipp, lyons_nazarov}, independent sets \cite{endreuj, E, V, hoppen},
$4$-regular spanning forests \cite{damien, kun}, colorings \cite{robin}.
Using this connection to random regular graphs in the reverse direction,
one can prove entropy inequalities \cite{invtree, bowen}
yielding necessary conditions for a process to be factor of i.i.d.;
this new tool has several applications already \cite{gamarnik, mustazee, mustazeebalint}.
Correlation bounds provide further necessary conditions \cite{cordec};
this is the main tool in the current paper as well. 
In \cite{spec} the possible ``correlation structures'' were described
for factor of i.i.d.\ processes by understanding their spectral measures.
See \cite{lyons} for futher references and for many open problems in the topic.

%Ezeket en kihagynam:
%On $\mathbb Z$, it was shown that every invariant process
%is a limit of factor of i.i.d.\ processes \cite{ornstein}. ...
%However, certain  combinatorial structures can be constructed
%with block factor of i.i.d.\ itself; see e.g.\ \cite{holroyd}. ...
%However, sevaral open questions remain; factor of i.i.d.\ constructions
%are often far from the necessary conditions, even in special family of processes
%like the Ising model \cite{lyons}.

\subsection*{Outline of the paper}
The rest of the paper is structured as follows.
In Section \ref{sec:2} we go through basic definitions,
present some examples, and show how Corollary \ref{cor:1-ended} follows
from Theorem \ref{thm:main}, the proof of which is given in Section \ref{sec:3}.
Finally, in Section \ref{sec:4} we prove our correlation decay result for directed edges
via bounding the norms of the powers of the non-backtracking operator.

%%%%%%%%%%%%%%%%%%%%%%%%%%%%%%%%%%%%%%%%%%%%%%%%%%%%%%%%%%%%%%%%%%%%%%%%%
%%%%%%%%%%%%%%%%%%%%%%%%%%%%%%%%%%%%%%%%%%%%%%%%%%%%%%%%%%%%%%%%%%%%%%%%%
%%%%%%%%%%%%%%%%%%%%%%%%%%%%%%%%%%%%%%%%%%%%%%%%%%%%%%%%%%%%%%%%%%%%%%%%%

\section{Preliminaries} \label{sec:2}

\subsection{Factors of i.i.d.}

Suppose that a group $\Ga$ acts on a countable set $S$.
Then $\Ga$ also acts on the space $M^S$ for a set $M$:
for any function $f$ defined on $S$ and for any $\ga \in \Ga$ let
\begin{equation} \label{eq:action}
(\ga \cdot f)(s) \defeq f( \ga^{-1} \cdot s) \quad \forall s \in S .
\end{equation}
In our setting $M$ is always equipped with a $\si$-algebra
(that is, $M$ is a measurable space).
As usual, the product space $M^S$ is equipped
with the smallest $\si$-algebra such that
all $M^S \to M$ coordinate projections are measurable.
This way a $\cdot \to M^S$ function is measurable
if and only if each coordinate function is measurable.
First we define the notion of factor maps.
\begin{definition}
Let $M_1,M_2$ be measurable spaces and
$S_1,S_2$ countable sets with a group $\Ga$ acting on both.
A measurable mapping $F \colon M_1^{S_1} \to M_2^{S_2}$ is
said to be a \emph{$\Ga$-factor} if it is $\Ga$-equivariant,
that is, it commutes with the $\Ga$-actions.

Whenever $\Ga$ acts transitively on $S_2$,
$F$ is determined by the function $f = \pi_o \circ F \colon M_1^{S_1} \to M_2$,
where $\pi_o \colon M_2^{S_2} \to M_2$ is the projection
corresponding to some distinguished element $o \in S_2$.
There is a one-to-one correspondence between
measurable $\Ga$-equivariant mappings $F \colon M_1^{S_1} \to M_2^{S_2}$
and measurable functions $f \colon M_1^{S_1} \to M_2$
that are invariant under the stabilizer of $o$.
%(Given such an $f$, one can get the corresponding $F$ by ...
%It is easy to see that $F$ will be well defined, measurable and $\Ga$-equivariant.)
\end{definition}
Next we explain what we mean by processes on $M^S$.
\begin{definition}
A probability measure on $M^S$ that is invariant under the $\Ga$-action
is called an \emph{invariant process}. The simplest examples are
\emph{i.i.d.\ processes}: take a probability measure $\mu$ on $M$
and consider the product measure $\nu = \mu^S$ on $M^S$.
(The $\Ga$-action on $(M^S, \mu^S)$ is
often called the \emph{generalized Bernoulli shift}.)
Given an i.i.d.\ process $\nu$ on $M_1^{S_1}$ and
a $\Ga$-factor $F \colon M_1^{S_1} \to M_2^{S_2}$,
the push-forward measure $F_{\ast} \nu$ is also $\Ga$-invariant.
Such processes are called \emph{factors} of the i.i.d.\ process $\nu$.
\end{definition}
Sometimes we will think of an invariant process $\nu$ on $M^S$
as an $M^S$-valued random variable (whose distribution is $\nu$),
or as a collection of $M$-valued random variables (whose joint distribution is $\nu$).
For example, if $Z_s$, $s\in S_1$ are independent, $M_1$-valued random variables
with some common distribution $\mu$, then $Z = \left( Z_s \right)_{s \in S_1}$
is an i.i.d.\ process on $M_1^{S_1}$.
Given a $\Ga$-factor $F \colon M_1^{S_1} \to M_2^{S_2}$,
$X \defeq F(Z)$ is a factor of the i.i.d.\ process $Z$.
Then $X$ is a collection $\left( X_s \right)_{s \in S_2}$ of $M_2$-valued random variables.
They can be expressed using the corresponding $f$ as well: clearly $X_o = f(Z)$,
and it is also easy to see that $X_s = f( \ga \cdot Z )$,
where $\ga \in \Ga$ is an arbitrary group element taking $o$ to $s$.

\subsection{Factors on $T_d$}
\label{sec:factors_Td}
In this paper we mainly consider the case
when $\Ga$ is the automorphism group $\Aut(T_d)$ of
the $d$-regular infinite tree $T_d$ ($d \geq 3$)
and $S$ is either the vertex set $V(T_d)$ or the directed edge set $E(T_d)$.
The latter consists of the ordered pairs $(u,v)$,
where $u,v \in V(T_d)$ are neighbors.
For a directed edge $e=(u,v)$, the inverse of $e$
is the directed edge $e^{-1} = (v,u)$.

In this setting, when we say \emph{factor of i.i.d.\ process},
we do not need to specify which i.i.d.\ process we have in mind.
The reason for this is that i.i.d.\ processes are factors of each other in this case.
The most natural i.i.d.\ processes to consider would be:
\begin{align*}
\nu_1: \quad & S=V(T_d); M=[0,1]; \mu \mbox{ is the Lebesgue measure,}\\
\nu_2: \quad & S=V(T_d); M=\{0,1, \ldots, m-1\} \mbox{ for some } m\geq 2;
\mu \mbox{ is the uniform measure on $M$,}\\
\nu_3: \quad & S=E(T_d); M=[0,1]; \mu \mbox{ is the Lebesgue measure,}\\
\nu_4: \quad & S=E(T_d); M=\{0,1, \ldots, m-1\} \mbox{ for some } m\geq 2;
\mu \mbox{ is the uniform measure on $M$.}
\end{align*}
It is trivial that $\nu_2$ is a factor of $\nu_1$
and that $\nu_4$ is a factor of $\nu_3$.
Extending the classical example of Ornstein and Weiss \cite{ornstein_weiss_1987},
it was shown in \cite{karen_ball} that $\nu_1$ is an $\Aut(T_d)$-factor of $\nu_2$.
(In \cite{karen_ball} the author considers the case $m=2$
but the same argument works for arbitrary $m \geq 2$.)
Furthermore, it can be seen easily that $\nu_2$ is a factor of $\nu_4$.
(Let the label of a vertex $v \in V(T_d)$ be
the sum of the labels of the directed edges starting at $v$, modulo $m$.)
Finally, it is possible to obtain $\nu_3$ as a factor of $\nu_1$.
(We can think of the $[0,1]$-label of a vertex as $d+1$ independent $[0,1]$-labels.
For each vertex $v$, we order its neighbors
according to their $(d+1)$th labels.
The directed edge going from $v$ to its $i$th
largest neighbor will get the $i$th label of $v$.)
These constructions show that the above i.i.d.\ processes
are all factors of each other. We will usually work with $\nu_1$.

\emph{Block factors} are $\Aut(T_d)$-factors obtained using a rule $f$ that
depends only on some finite-radius ball around the distinguished vertex $o \in V(T_d)$.
If this radius is $0$, then we simply have $X_v = \varphi( Z_v ), \forall v \in V(T_d)$
for some measurable $\varphi \colon M_1 \to M_2$. If this is the case, 
then we say that $X$ is a \emph{pointwise factor} of $Z$. 
If $Z$ is an i.i.d.\ process, then so is any pointwise factor of $Z$. 
However, as the next example shows,
there exists a factor of i.i.d.\ process $X$ that is ``universal'' in the sense
that any factor of i.i.d.\ process can be obtained as a pointwise factor of $X$.
\begin{example} \label{ex}
Let $Z$ denote the i.i.d.\ process on $[0,1]^{V(T_d)}$ (of distibution $\nu_1$).
We claim that there exists a factor $X$ of $Z$
on $[0,1]^{V(T_d)}$ with the following properties.
\begin{enumerate}[(a)]
	\item With probability $1$ for any $u,v \in V(T_d)$ the values of $X_u, X_v$
	determine (in a measurable way) the values of $X$ along the whole path connecting $u$ and $v$.
	In other words, for any $V \subset V(T_d)$,
	$X_v, v \in V$ determine $X$ on the convex hull of $V$.
	\item Any factor $X'$ of $Z$ on some $M^{V(T_d)}$ can be obtained
	as the pointwise factor of $X$, that is,
	$X'_v = \varphi(X_v), \forall v \in V(T_d)$ for some measurable $\varphi \colon [0,1] \to M$.
\end{enumerate}
\end{example}
\begin{proof}
The idea is to encode the whole labelled tree in $X_v$.
To do this with an $\Aut(T_d)$-factor we need to do the encoding
in a way that it only contains the isomorphism type of the labelled tree rooted at $v$.
With probability $1$ the values $Z_v$ are pairwise distinct.
Then for a given configuration $\omega = \left( \omega_u \right)_{u \in V(T_d)}$
of pairwise disjoint labels we assign the following sequence to any given vertex $v$:
$$ (
\underbrace{\omega_v}_{\parbox{0.5in}{\tiny label of $v$}},
\underbrace{\omega_{v_1} < \ldots < \omega_{v_d}}_
{\parbox{1in}{\tiny labels of the\\neighbors of $v$\\in increasing order}},
\underbrace{\omega_{v_{1,1}} < \ldots < \omega_{v_{1,d-1}}}_
{\parbox{1in}{\tiny labels of the\\remaining $d-1$\\neighbors of $v_1$}},
\underbrace{\omega_{v_{2,1}} < \ldots < \omega_{v_{2,d-1}}}_
{\parbox{1in}{\tiny labels of the\\remaining $d-1$\\neighbors of $v_2$}},
\ldots ) .$$
Finally, we apply a fixed injective measurable $[0,1]^{\IN} \to [0,1]$ mapping
(whose inverse is also measurable) to any such sequence to get the new label $\al_v$ of $v$.

By knowing the labels $\al_u$ and $\al_v$ of the vertices $u$ and $v$,
we know the isomorphism type of the $\omega$-labelled tree,
and also the original labels $\omega_u, \omega_v$ of $u$ and $v$.
Then to find out the $\al$-label of any other vertex,
it suffices to know its original $\omega$-label.
Now let $n$ denote the distance of $u$ and $v$.
Then for any integer $0 < k < n$, the $k$-ball around $u$ and the $(n-k)$-ball around $v$
have exactly one common vertex (the $k$th vertex on the $u$-$v$ path).
Therefore there is one common value
among the $\omega$-labels of these two balls, the label of the common vertex.
Thus $\alpha_u$ and $\alpha_v$ indeed determine the $\omega$-label
and hence the $\alpha$-label of any vertex on the $u$-$v$ path, which proves (a).

To see (b), recall that any factor $X'$ of $Z$ comes from a rule $f \colon [0,1]^{V(T_d)} \to M$
that is measurable and invariant under the stabilizer of $o$.
Clearly, such an $f$ can be obtained
as a measurable function $\varphi$ of the $\alpha$-label at $o$.
\end{proof}
We claim that Theorem \ref{thm:main} is essentially sharp for any $V_1, V_2 \subset V(T_d)$.
Let $X$ be a process satisfying the properties (a) and (b) above,
and let us pick $v_1$ and $v_2$ in the convex hulls of $V_1$ and $V_2$
with the smallest possible distance: $\dist(v_1, v_2) = k$.
It follows from (a) that $X_{v_i}$ is a measurable function of $X_u, u \in V_i$,
and by (b) we can obtain $X'_{v_i}$  as a measurable function of $X_{v_i}$
for any factor $X'$ of $Z$. Thus we can choose measurable $h_i$ in a way
that $\corr( h_1(X), h_2(X) ) = \corr( X'_{v_1}, X'_{v_2} )$
for any given factor of i.i.d.\ process $X'$ on $\IR^{V(T_d)}$.
Therefore the fact that the vertex-correlation bound \eqref{eq:corr_decay_for_vertices}
is sharp means that \eqref{eq:corr_of_subgraphs}  is also essentially sharp.
(In \cite{cordec} the authors give a very simple example for which
the correlation of two vertices of distance $k$ is of order $1 / (\sqrt{d-1})^k$,
which already shows that \eqref{eq:corr_decay_for_vertices} is essentially sharp,
but it is clear from their proof of the bound that it actually has to be sharp.)

\subsection{Tail $\si$-algebras}

For $v \in V(T_d)$ let $\pi_v \colon M^{V(T_d)} \to M$
denote the natural coordinate projection.
For $V \subseteq V(T_d)$ let $\si(V)$ be
the $\si$-algebra generated by the maps $\pi_v$, $v \in V$.
\begin{definition}
The \emph{tail $\si$-algebra} is defined as $\bigcap_r \si( V(T_d) \setminus B_r )$,
where $B_r$ stands for the $r$-ball around some fixed vertex $o$. 
Clearly, the tail does not depend on the choice of $o$. 

A $\si$-algebra is said to be \emph{trivial} w.r.t.\ a probability measure
if it contains only sets of measure $0$ or $1$. 

We say that a process \emph{has full tail} if its tail is 
the whole $\si$-algebra $\si( V(T_d) )$. 
\end{definition}
It is open whether trivial tail implies factor of i.i.d. 
As for the reverse direction, it follows easily 
from the Kolmogorov $0$-$1$ Law that block factors have trivial tail. 
This is not true for arbitrary factors, though. 
In \cite[Proposition 2.4]{lyons} it was shown that 
``the uniform random perfect matching on $T_d$'' has full tail 
since knowing the matching outside a ball determines it inside as well. 
(It had been known earlier by Lyons and Nazarov \cite{lyons_nazarov}
that this process is factor of i.i.d.) 

In Example \ref{ex} we presented a factor of i.i.d.\ process on $[0,1]^{V(T_d)}$ for which
$\si( V )$ coincides with the $\si$ of the convex hull of $V$ for any $V \subset V(T_d)$.
Such a process clearly has a full tail. 
One could consider some other sequence of shrinking subsets of $V(T_d)$ with empty intersection: 
$V(T_d) \supset V_1 \supset V_2 \supset \ldots$ with $\bigcap_n V_n = \emptyset$, 
and define another notion of tail by considering the $\si$-algebra $\bigcap_n \si( V_n )$. 
The only case for which such a tail could be trivial for Example \ref{ex} 
is when the convex hulls of $V_n$ ``converge to infinity''. 
The only such tails are the 1-ended tails introduced in \cite{pemantle} 
under the name ``one-sided tails'', see also \cite[Section 2]{lyons}. 
\begin{definition}
The \emph{1-ended tail $\si$-algebra} corresponding
to an infinite simple path $(v_0, v_1, v_2, \ldots )$ is $\bigcap_n \si( D_n )$,
where $D_n$ is the set of vertices closer to $v_n$ than to $v_{n-1}$.

It is easy to see that for an $\Aut(T_d)$-invariant measure on $M^{V(T_d)}$
the 1-ended tails are all trivial or none are trivial.
\end{definition}
The group $\Aut^{+}(T_d)$ of ``parity-preserving'' automorphisms
is a subgroup of $\Aut(T_d)$ of index $2$. (The vertices of $T_d$ can be 
partitioned into two classes based on the parity of their distance to a fixed vertex, 
and an automorphism either takes each class into itself, 
i.e.\ \emph{preserves parity}, or swaps the two classes.) 
In \cite{pemantle} Pemantle showed that any $\Aut^{+}(T_d)$-invariant process that is ergodic 
has trivial $1$-ended tails, which in turn implies mixing, a weak law of large numbers, 
and a version of the Birkhoff Ergodic Theorem on $T_d$. 

Next we show how the 1-ended tail triviality 
(for the case of factor of i.i.d.\ processes) 
follows from Theorem \ref{thm:main}.
\begin{proof}[Proof of Corollary \ref{cor:1-ended}]
Let $A_1$ be an arbitrary event in the 1-ended tail
and $A_2 \in \si(V_2)$ for some finite set $V_2 \subset V(T_d)$.
We claim that $A_1$ and $A_2$ are independent.
To see this, it suffices to show that the indicator functions
$h_i \defeq \ind_{A_i}$, $i=1,2$ have zero correlation.
However, for any fixed $n$ we can apply Theorem \ref{thm:main}
for $V_1 = D_n$ since $h_1 = \ind_{A_1}$ can be considered
as a measurable $M^{D_n} \to \IR$ function for any $n$.
As $n \to \infty$ the distance of $D_n$ and the convex hull of $V_2$ gets arbitrarily large,
so we obtain that the correlation must be $0$.

Therefore $A_1$ is independent from the $\si$-algebra $\si(V_2)$
for any finite subset $V_2$ of the vertex set.
The Dynkin $\pi$-$\lambda$ lemma implies that $A_1$ is independent
from the whole $\si$-algebra $\si( V(T_d) )$. In particular,
$A_1$ is independent from itself, so its probability is either 0 or 1.
This holds for any event in the 1-ended tail meaning that it is trivial.
\end{proof}
\begin{remark} \label{rm:eq}
The following are equivalent for any $\Aut(T_d)$-invariant process $\mu$ on $M^{V(T_d)}$.
\begin{enumerate}[(a)]
\item The $1$-ended tails are trivial w.r.t.\ $\mu$.
\item Let $A$ denote an event on $M^{V(T_d)}$ depending only on the coordinates $U \subset V(T_d)$ 
(that is, $A = A' \times M^{V(T_d) \setminus U}$ for some measurable subset $A'$ of $M^U$). 
Furthermore, let the events $B_n$ depend on the coordinates $V_n$. 
Then $B_n$ is asymptotically independent from $A$ (i.e.\ $\mu(A \cap B_n) - \mu(A) \mu(B_n) \to 0$) 
whenever the distance of the convex hulls of $U$ and $V_n$ goes to infinity as $n \to \infty$. 
\end{enumerate}
The proof above essentially shows that (b) implies (a), 
while the proof of \cite[Corollary 7]{pemantle} yields the other implication. 
\end{remark}
%

%%%%%%%%%%%%%%%%%%%%%%%%%%%%%%%%%%%%%%%%%%%%%%%%%%%%%%%%%%%%%%%%%%%%%%%%%
%%%%%%%%%%%%%%%%%%%%%%%%%%%%%%%%%%%%%%%%%%%%%%%%%%%%%%%%%%%%%%%%%%%%%%%%%
%%%%%%%%%%%%%%%%%%%%%%%%%%%%%%%%%%%%%%%%%%%%%%%%%%%%%%%%%%%%%%%%%%%%%%%%%

\section{Proof of the main result} \label{sec:3}

Recall that $\Tt$ denotes the rooted $(d-1)$-ary tree:
the degree of each vertex of $\Tt$ is $d$ except for
one vertex (the root) of degree $d-1$ (that is, every vertex has $d-1$
``offsprings''). If we delete an edge of $T_d$,
both connected components will be isomorphic to $\Tt$.

\textbf{Step 1.}
Let $V_1$ and $V_2$ be subsets of $V(T_d)$ and
let $k$ denote the distance of their convex hulls as in Theorem \ref{thm:main}.
This means that there exist unique vertices $v_1, v_2$
such that $v_i$ is in the convex hull of $V_i$ and
there is a (unique) path of length $k$ connecting $v_1$ and $v_2$ (we assume $k \geq 1$).
If we delete the edges of this path,
then $V_i$ is contained in the connected component of $v_i$,
which is isomorphic to $\Tt$, $i=1,2$, see Figure \ref{fig1}. 
Let us replace $V_i$ with its component. This way $V_1$ and $V_2$
get larger while $k$ remains the same, therefore
it suffices to prove Theorem \ref{thm:main} in this case.
\begin{figure}[ht]
\centering
\def\svgwidth{0.5\columnwidth}
\begingroup%
  \makeatletter%
  \providecommand\color[2][]{%
    \errmessage{(Inkscape) Color is used for the text in Inkscape, but the package 'color.sty' is not loaded}%
    \renewcommand\color[2][]{}%
  }%
  \providecommand\transparent[1]{%
    \errmessage{(Inkscape) Transparency is used (non-zero) for the text in Inkscape, but the package 'transparent.sty' is not loaded}%
    \renewcommand\transparent[1]{}%
  }%
  \providecommand\rotatebox[2]{#2}%
  \ifx\svgwidth\undefined%
    \setlength{\unitlength}{547.999992bp}%
    \ifx\svgscale\undefined%
      \relax%
    \else%
      \setlength{\unitlength}{\unitlength * \real{\svgscale}}%
    \fi%
  \else%
    \setlength{\unitlength}{\svgwidth}%
  \fi%
  \global\let\svgwidth\undefined%
  \global\let\svgscale\undefined%
  \makeatother%
  \begin{picture}(1,0.59754696)%
    \put(0,0){\includegraphics[width=\unitlength,page=1]{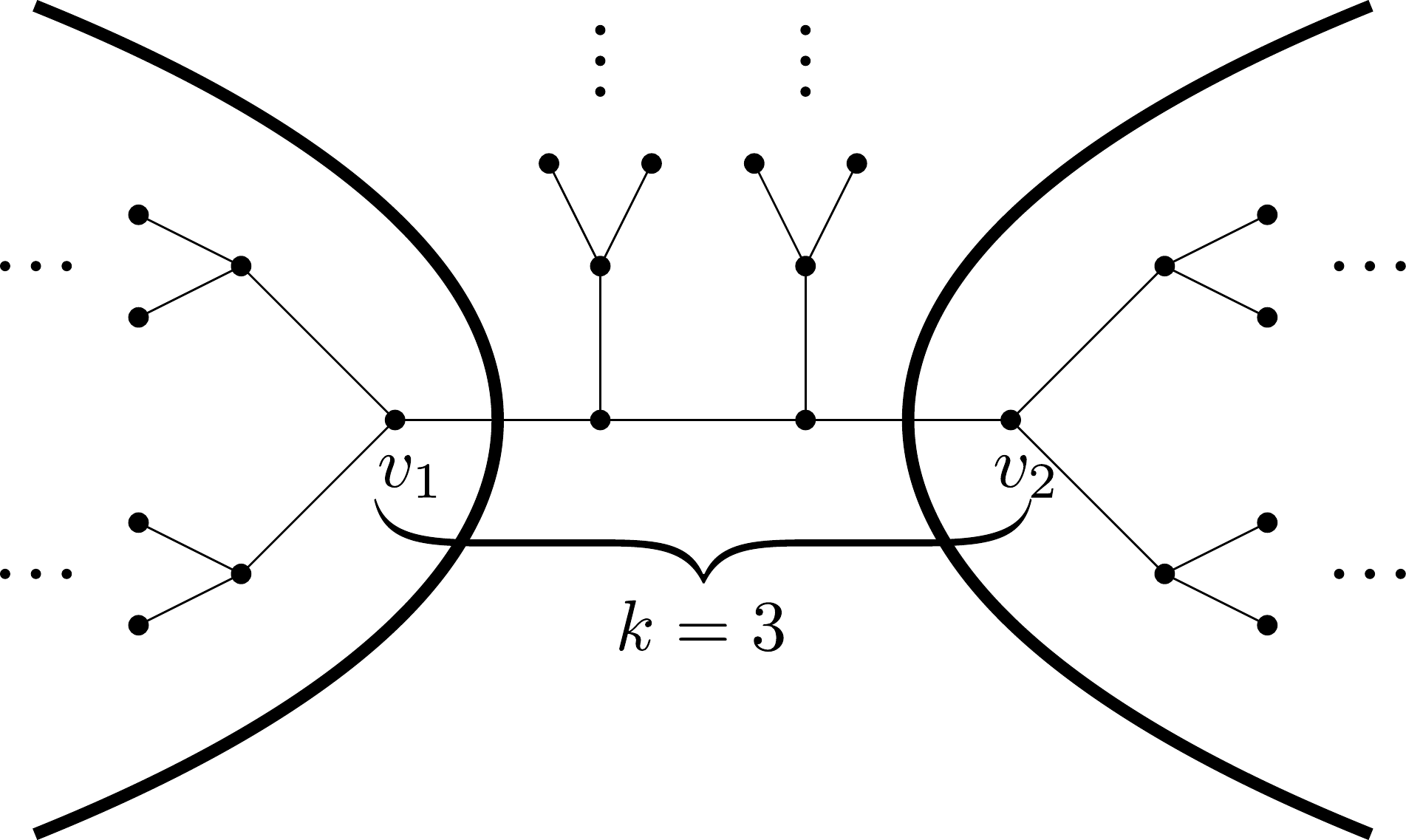}}%
  \end{picture}%
\endgroup%
\caption{Step 1: we might assume that $V_1$ and $V_2$ are $(d-1)$-ary subtrees of distance $k$}
\label{fig1}
\end{figure}

\textbf{Step 2.}
So from this point on we will assume that $V_1$ and $V_2$ are
disjoint copies of $\Tt$ in $T_d$, with their roots at distance $k$.
This means that $\left( X_v \right)_{v \in V_1}$ and $\left( X_v \right)_{v \in V_2}$ 
can be viewed as processes on $V(\Tt)$. 
More precisely, we define $M^{V(\Tt)}$-valued random variables $X_1, X_2$ 
by fixing graph isomorphisms $\Phi_i \colon V(\Tt) \to V_i$ for $i=1,2$, 
and setting $(X_i)_u \defeq X_{\Phi_i(u)}$ for $i=1,2$, $u \in V(\Tt)$. 
Using the $\Aut(T_d)$-invariance of $X$
the following properties of $X_1$ and $X_2$ follow easily.
\begin{claim}
Both $X_1$ and $X_2$ are invariant under $\Gat \defeq \Aut(\Tt)$.
In fact, $( \ga_1 \cdot X_1, \ga_2 \cdot X_2 )$ and $(X_1,X_2)$
have the same joint distribution for any $\ga_1, \ga_2 \in \Gat$.
Furthermore, $(X_1,X_2)$ and $(X_2,X_1)$ also have the same joint distribution.
\end{claim}
Now let $h_1$, $h_2$ be as in Theorem \ref{thm:main}.
Then there clearly exist measurable functions $f_1, f_2 \colon M^{V(\Tt)} \to \IR$
such that $f_i(X_i) = h_i(X)$ ($i=1,2$). According to the following lemma
we might assume that $f_1$ and $f_2$ are actually the same function.
\begin{lemma} \label{lem:2functions}
Let $(A, \mathcal{F} )$ be an arbitrary measurable space.
Suppose that the $(A, \mathcal{F} )$-valued random variables $X_1,X_2$ are exchangeable
(that is, $(X_1,X_2)$ and $(X_2,X_1)$ have the same joint distribution),
and that there exists a constant $\al \geq 0$ with the property
that for any measurable $f \colon A \to \IR$ we have
\begin{equation} \label{eq:1function}
\left| \corr\big( f(X_1), f(X_2) \big) \right| \leq \al
\mbox{ provided that $f(X_1)$ has finite variance.}
\end{equation}
Then for any measurable functions $f_1,f_2 \colon A \to \IR$
\begin{equation} \label{eq:2functions}
\left| \corr\big( f_1(X_1), f_2(X_2) \big) \right| \leq \al
\mbox{ provided that $f_1(X_1)$ and $f_2(X_2)$ have finite variances.}
\end{equation}
\end{lemma}
\begin{proof}
We might assume that $\var( f_1(X_1) ) = \var( f_2(X_2) ) = 1$.
(If one of the variances is $0$, then
the correlation is $0$ by definition
and the statement of the lemma holds trivially.
Otherwise we can rescale $f_1$ and $f_2$ to make the variances equal to $1$
without changing the correlation.)

Since $X_1$ and $X_2$ are exchangeable we have
$ \cov( f_1(X_1), f_2(X_2) ) = \cov( f_1(X_2), f_2(X_1) )$.
It follows that
\begin{multline*}
\corr\big( f_1(X_1), f_2(X_2) \big) = \cov\big( f_1(X_1), f_2(X_2) \big) \\
= \frac{1}{4} \Big( \cov\big( (f_1+f_2)(X_1) , (f_1+f_2)(X_2) \big)
- \cov\big( (f_1-f_2)(X_1) , (f_1-f_2)(X_2) \big) \Big) .
\end{multline*}
Using the triangle inequality and applying \eqref{eq:1function}
to the function $f= f_1+f_2$ and to $f=f_1-f_2$ we obtain that
\begin{multline*}
\left| \corr\big( f_1(X_1), f_2(X_2) \big) \right| \leq \frac{\al}{4} \Big(
\var\big( (f_1+f_2)(X_1) \big) + \var\big( (f_1-f_2)(X_1) \big) \Big) \\
= \frac{\al}{4} \Big( 2\var\big( f_1(X_1) \big) + 2\var\big( f_2(X_1) \big) \Big) = \al .
\end{multline*}
\end{proof}
\textbf{Step 3.}
It remains to bound $\corr( f(X_1), f(X_2) )$ 
for any given measurable function $f \colon M^{V(\Tt)} \to \IR$. 
We claim that it suffices to do this
in the case when $f$ is $\Gat = \Aut(\Tt)$-invariant.
The idea is to ``average $f$ over the orbits of the $\Gat$-action'',
that is, take the function
$\bar{f}( \omega ) = \int_{\Gat} f(\ga \cdot \omega) \, \mathrm{d} \gamma$ instead.

To make this more precise,
let us consider the natural topology and the \emph{Haar measure} on $\Gat$.
(For $u,v \in V(\Tt)$ let $\Gat_{u,v}$ denote the set of
those $\ga \in \Gat$ that take $u$ to $v$.
Then the sets $\Gat_{u,v}$ form a base of the topology.
Furthermore, $\Gat_{u,v}$ is non-empty if and only if
$u$ and $v$ have the same distance $n$ from the root of $\Tt$,
in which case the measure of $\Gat_{u,v}$ is $1/(d-1)^n$.)
By $\mu$ we will denote the common distribition of $X_1$ and $X_2$,
so $\mu$ is a probability measure on $M^{V(\Tt)}$.
It can be seen easily that the following
$\Gat \times M^{V(\Tt)} \to \IR$ function is measurable:
$$ \varphi(\ga, \omega) \defeq f( \ga \cdot \omega ) \quad \ga \in \Gat, \omega \in M^{V(\Tt)}, $$
and that $|\varphi|$ has finite integral. Therefore
$$ \bar{f}( \omega ) \defeq \int_{\Gat} \varphi(\ga, \omega) \, \mathrm{d} \gamma =
\int_{\Gat} f(\ga \cdot \omega) \, \mathrm{d} \gamma $$
is defined for $\mu$-a.e.\ $\omega \in M^{V(\Tt)}$ and it is measurable.
Furthermore, $\bar{f} \colon M^{V(\Tt)} \to \IR$ is clearly $\Gat$-invariant.
(To be more precise, there exists a measurable, $\Gat$-invariant
$M^{V(\Tt)} \to \IR$ function that is equal to $\bar{f}$ $\mu$-a.e.)

Although we will not need this, we mention that there is another way to define $\bar{f}$:
take the $\si$-algebra of $\Gat$-invariant measurable sets in $M^{V(\Tt)}$
and let $\bar{f}$ be the conditional expectation of $f$ w.r.t.\ this $\si$-algebra.
\begin{proposition}
The function $\bar{f}$ has the following properties.
\begin{enumerate}[(a)]
	\item $\E \bar{f}(X_i) = \E f(X_i)$,
	\item $\E (\bar{f})^2(X_i) \leq \E f^2(X_i)$,
	\item $\E \bar{f}(X_1)\bar{f}(X_2) = \E f(X_1)f(X_2)$. 	
\end{enumerate}
\end{proposition}
\begin{proof}
If $\omega$ is $\mu$-random element of $M^{V(\Tt)}$,
then the distribution of $\ga \cdot \omega$ is also $\mu$ for any fixed $\ga \in \Gat$.
Using this fact and Fubini's theorem (a) easily follows:
$$ \E \bar{f}(X_i) = \int \bar{f} \, \mathrm{d} \mu =
\int \int_{\Gat} f(\ga \cdot \omega) \, \mathrm{d} \gamma \mathrm{d} \mu =
\int_{\Gat} \underbrace{\int f(\ga \cdot \omega) \, \mathrm{d} \mu}_{\E f(X_i)} \, \mathrm{d} \gamma = \E f(X_i) .$$
To see (b) we need to first use the Cauchy-Schwarz inequality
before applying Fubini's theorem to $\varphi^2$:
\begin{multline*}
\E (\bar{f})^2(X_i) = \int (\bar{f})^2 \, \mathrm{d} \mu =
\int \left( \int_{\Gat} f(\ga \cdot \omega) \, \mathrm{d} \gamma \right)^2 \mathrm{d} \mu \leq
\int \int_{\Gat} f^2(\ga \cdot \omega) \, \mathrm{d} \gamma \mathrm{d} \mu \\
 = \int_{\Gat} \underbrace{\int f^2(\ga \cdot \omega) \, \mathrm{d} \mu}_{\E f^2(X_i)}
 \, \mathrm{d} \gamma = \E f^2(X_i) .
\end{multline*}
Finally, to prove (c) we use that
$(\ga_1 \cdot X_1, \ga_2 \cdot X_2)$ and $(X_1,X_2)$ have the same joint distribution
for any fixed $\ga_1,\ga_2 \in \Gat$. It follows that
$$ \E f(\ga_1 \cdot X_1)f(\ga_2 \cdot X_2) = \E f(X_1)f(X_2) .$$
Integrating this equality w.r.t.\ $\mathrm{d} \gamma_1 \mathrm{d} \gamma_2$
and using Fubini's theorem once again
we conclude that $\E \bar{f}(X_1)\bar{f}(X_2) = \E f(X_1)f(X_2)$.
(Each time we used it, the conditions of Fubini's theorem were satisfied
as $|\varphi|$ and $\varphi^2$ are measurable and have finite integrals.)
\end{proof}
It follows from (a) and (c) that $\bar{f}(X_1)$ and $\bar{f}(X_2)$
have the same covariance as $f(X_1)$ and $f(X_2)$,
while (a) and (b) imply that $\var \bar{f}(X_i) \leq \var f(X_i)$.
Consequently,
$$ \left| \corr \big( f(X_1), f(X_2) \big) \right| \leq
\left| \corr \big( \bar{f}(X_1), \bar{f}(X_2) \big) \right| .$$
Therefore it suffices to bound the correlation for $\bar{f}$,
that is, we might assume that $f$ was $\Gat$-invariant in the first place.

\textbf{Step 4.}
Whenever we have a factor of i.i.d.\ process $X_v$ on $M^{V(T_d)}$
and a measurable $\Gat$-invariant function $f \colon M^{V(\Tt)} \to \IR$,
we can combine them to create a factor of i.i.d.\ process $Y_e$ on $\IR^{E(T_d)}$.
To get $Y_e$ for a directed edge $e = (u,w)$ we take the subtree $T_e$ ``behind'' $e$
(that is, $V(T_e)$ consists of those vertices of $T_d$ that are closer to $u$ than to $w$),
and apply $f$ to $\left( X_v \right)_{v \in V(T_e)}$, see Figure \ref{fig2}. 
We can do this because $T_e$ is isomorphic to $\Tt$,
and $Y_e$ will be well defined since $f$ is $\Gat$-invariant.
It is also easy to see that $\left( Y_e \right)_{e \in E(T_d)} $
will be a factor of i.i.d.\ process.

Furthermore, $Y_{e_i} = f(X_i)$, $i=1,2$, where
$e_1$ is the directed edge starting at $v_1$ and ``pointing towards'' $v_2$, and
$e_2$ is the directed edge starting at $v_2$ and ``pointing towards'' $v_1$.
So it remains to show that the correlation of $Y_{e_1}$ and $Y_{e_2}$ 
is small if the distance of $e_1$ and $e_2$ is large. 
This final step will be done in the next section, 
see Theorem \ref{thm:edge_corr} below. 
The bound \eqref{eq:edge_corr} clearly implies Theorem \ref{thm:main}. 
(Note that $\dist(e_1,e_2) = k-1$ in our case.) 
\begin{figure}[ht]
\centering
\def\svgwidth{0.5\columnwidth}
\begingroup%
  \makeatletter%
  \providecommand\color[2][]{%
    \errmessage{(Inkscape) Color is used for the text in Inkscape, but the package 'color.sty' is not loaded}%
    \renewcommand\color[2][]{}%
  }%
  \providecommand\transparent[1]{%
    \errmessage{(Inkscape) Transparency is used (non-zero) for the text in Inkscape, but the package 'transparent.sty' is not loaded}%
    \renewcommand\transparent[1]{}%
  }%
  \providecommand\rotatebox[2]{#2}%
  \ifx\svgwidth\undefined%
    \setlength{\unitlength}{397.999992bp}%
    \ifx\svgscale\undefined%
      \relax%
    \else%
      \setlength{\unitlength}{\unitlength * \real{\svgscale}}%
    \fi%
  \else%
    \setlength{\unitlength}{\svgwidth}%
  \fi%
  \global\let\svgwidth\undefined%
  \global\let\svgscale\undefined%
  \makeatother%
  \begin{picture}(1,0.82275311)%
    \put(0,0){\includegraphics[width=\unitlength,page=1]{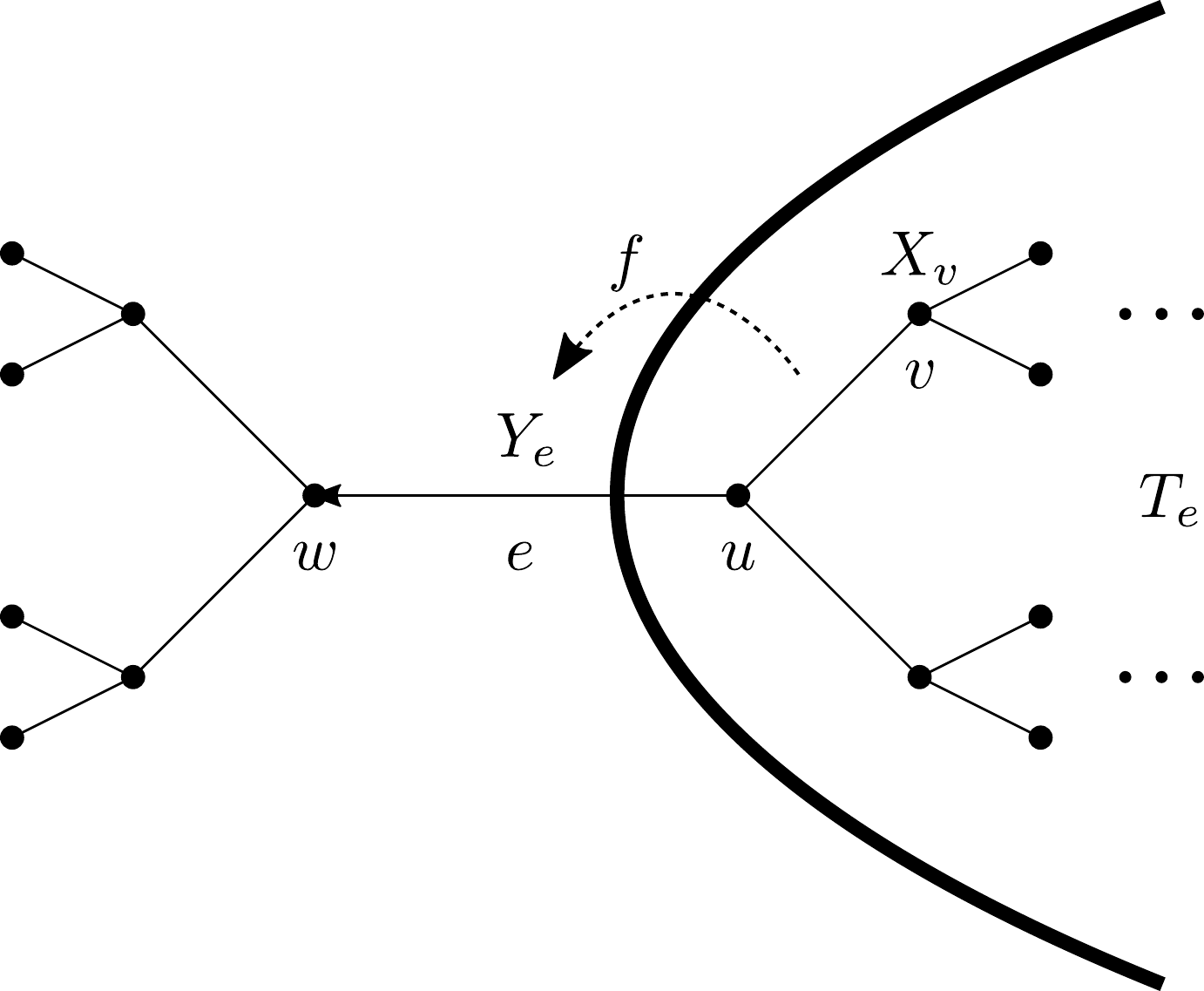}}%
  \end{picture}%
\endgroup%
\caption{Step 4: obtaining $Y_e$ from $\left( X_v \right)_{v \in V(T_e)}$}
\label{fig2}
\end{figure}
%

%%%%%%%%%%%%%%%%%%%%%%%%%%%%%%%%%%%%%%%%%%%%%%%%%%%%%%%%%%%%%%%%%%%%%%%%%
%%%%%%%%%%%%%%%%%%%%%%%%%%%%%%%%%%%%%%%%%%%%%%%%%%%%%%%%%%%%%%%%%%%%%%%%%
%%%%%%%%%%%%%%%%%%%%%%%%%%%%%%%%%%%%%%%%%%%%%%%%%%%%%%%%%%%%%%%%%%%%%%%%%

\section{Correlation decay for directed edges} \label{sec:4}

In \cite{cordec} Backhausz, Szegedy and Vir\'ag bounded the correlation
of a pair of vertices for factor of i.i.d.\ processes on $\IR^{V(T_d)}$,
see \eqref{eq:corr_decay_for_vertices}.
The goal of this section is to prove a similar bound
but for directed edges instead of vertices, that is,
for factor of i.i.d.\ processes on $\IR^{E(T_d)}$.

By the distance of two undirected edges $e_1,e_2$ we mean the smallest integer $k$
for which there exists a path containing $k+1$ edges including $e_1$ and $e_2$, see Figure \ref{fig3}. 
As for the distance of directed edges, we simply forget the directions of the edges
and take the distance of the corresponding undirected edges. Equivalently,
$$ \dist\left( (u_1,u_2) ; (v_1,v_2) \right) \defeq 
\begin{cases}
0 \text{, if } u_1=v_1; u_2=v_2 \text{ or } u_1=v_2; u_2=v_1,\\
1 + \min_{i,j \in \{1,2\} } \dist( u_i, v_j ) \text{, otherwise.}
\end{cases}
$$
\begin{figure}[ht]
\centering
\def\svgwidth{0.35\columnwidth}
\begingroup%
  \makeatletter%
  \providecommand\color[2][]{%
    \errmessage{(Inkscape) Color is used for the text in Inkscape, but the package 'color.sty' is not loaded}%
    \renewcommand\color[2][]{}%
  }%
  \providecommand\transparent[1]{%
    \errmessage{(Inkscape) Transparency is used (non-zero) for the text in Inkscape, but the package 'transparent.sty' is not loaded}%
    \renewcommand\transparent[1]{}%
  }%
  \providecommand\rotatebox[2]{#2}%
  \ifx\svgwidth\undefined%
    \setlength{\unitlength}{353.67212069bp}%
    \ifx\svgscale\undefined%
      \relax%
    \else%
      \setlength{\unitlength}{\unitlength * \real{\svgscale}}%
    \fi%
  \else%
    \setlength{\unitlength}{\svgwidth}%
  \fi%
  \global\let\svgwidth\undefined%
  \global\let\svgscale\undefined%
  \makeatother%
  \begin{picture}(1,0.16382758)%
    \put(0,0){\includegraphics[width=\unitlength,page=1]{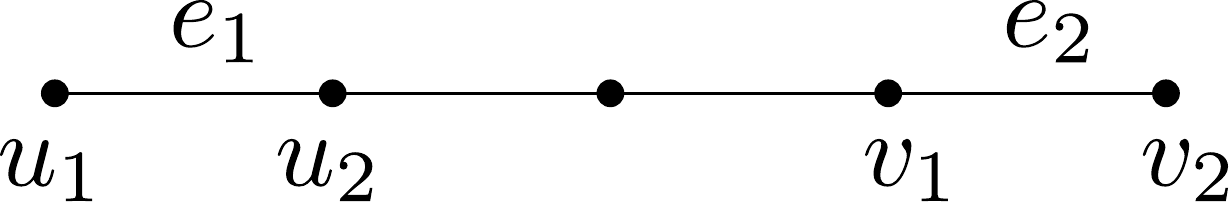}}%
  \end{picture}%
\endgroup%
\caption{The distance of the edges $e_1$ and $e_2$ in the figure is $3$}
\label{fig3}
\end{figure}
\begin{theorem} \label{thm:edge_corr}
Let $Y = \left( Y_{e} \right)_{e \in E(T_d)}$ be a
factor of i.i.d.\ process on $\IR^{E(T_d)}$. Then
\begin{equation} \label{eq:edge_corr}
 \left| \corr( Y_{e_1}, Y_{e_2} ) \right| \leq
(k+1) \left( \frac{1}{ \sqrt{d-1} } \right)^{k-1}
\mbox{, where } k = \dist(e_1,e_2),
\end{equation}
provided that $\var Y_e < \infty$.
\end{theorem}

It would be possible to start with
the vertex-correlation bound \eqref{eq:corr_decay_for_vertices}
and deduce a somewhat weaker version of \eqref{eq:edge_corr} from that.
This would involve some tedious calculations, however.
Instead, we will apply similar ideas as in \cite{cordec},
where the norms of certain polynomials of the adjacency operator 
were determined to obtain the bound \eqref{eq:corr_decay_for_vertices}. 
Here we will need to work with 
the non-backtracking operator instead of the adjacency operator.

The key observation is that $\E Y_{e_1} Y_{e_2} $ can be expressed
as an inner product on the $L^2$ space over $[0,1]^{V(T_d)}$.
Some kind of a non-backtracking operator can be defined on this space,
and to bound the inner product in question one needs to determine
the norm of the $k$-th power of this operator.
This, however, can be traced back to the case
of the ordinary non-backtracking operator $B$ on $T_d$.

\subsection{The non-backtracking operator on $T_d$}
\label{sec:4.1} 

For an undirected simple graph $G$ let $V(G)$ and $E(G)$ be
the vertex set and the directed edge set of $G$, respectively.
(We assume that $G$ is locally finite.)
For $e,e' \in E(G)$ we write $e \to e'$
if $e = (u,v)$ and $e'=(v,w)$ for some $u,v,w \in V(G)$ with $u \neq w$,
that is, if the head of $e$ coincides with the tail of $e'$ and $e' \neq e^{-1}$.
If this is the case, then we say that $e$ is the \emph{predecessor} of $e'$,
and $e'$ is the \emph{successor} of $e$, see Figure \ref{fig4}. 
\begin{figure}[ht]
\centering
\def\svgwidth{0.25\columnwidth}
\begingroup%
  \makeatletter%
  \providecommand\color[2][]{%
    \errmessage{(Inkscape) Color is used for the text in Inkscape, but the package 'color.sty' is not loaded}%
    \renewcommand\color[2][]{}%
  }%
  \providecommand\transparent[1]{%
    \errmessage{(Inkscape) Transparency is used (non-zero) for the text in Inkscape, but the package 'transparent.sty' is not loaded}%
    \renewcommand\transparent[1]{}%
  }%
  \providecommand\rotatebox[2]{#2}%
  \ifx\svgwidth\undefined%
    \setlength{\unitlength}{338.91660588bp}%
    \ifx\svgscale\undefined%
      \relax%
    \else%
      \setlength{\unitlength}{\unitlength * \real{\svgscale}}%
    \fi%
  \else%
    \setlength{\unitlength}{\svgwidth}%
  \fi%
  \global\let\svgwidth\undefined%
  \global\let\svgscale\undefined%
  \makeatother%
  \begin{picture}(1,0.19048061)%
    \put(0,0){\includegraphics[width=\unitlength,page=1]{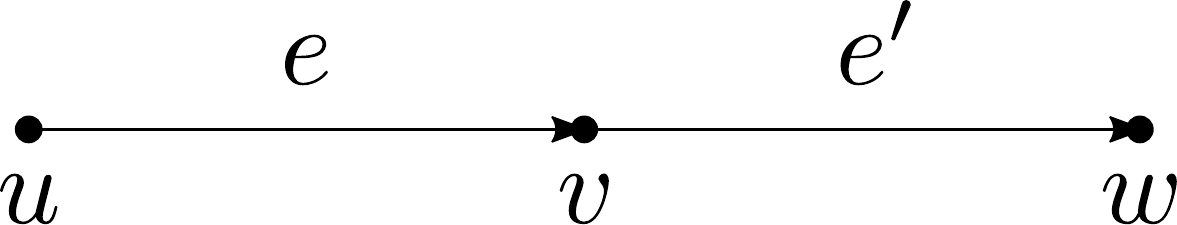}}%
  \end{picture}%
\endgroup%
\caption{$e=(u,v)$ is the predecessor of $e'=(v,w)$}
\label{fig4}
\end{figure}

By \emph{a non-backtracking walk} of length $k$ we mean a sequence of
directed edges $e_0, e_1, \ldots, e_k$ such that $e_i \to e_{i+1}$, $i=0,\ldots,k-1$.
If there exists a non-backtracking walk of length $k$ from $e$ to $e'$,
we write $e \to_k e'$.

There is a corresponding operator on $\ell^2( E(G) )$
called the \emph{non-backtracking operator}.
It is usually denoted by $B = B_G$, and is defined by
\begin{equation} \label{eq:nbt}
(Bf)(e) \defeq \sum_{e' \to e} f(e'), \quad f \in \ell^2( E(G) ), e \in E(G) .
\end{equation}
For regular graphs its spectrum is closely related to that of the adjacency operator,
but the non-backtracking operator often proves to be a more efficient tool, 
for example in understanding the spectral gap and expansion properties of 
random regular graphs \cite{charles, friedman, doron}. 
In \cite{alon}, and more recently, in \cite{anna, kempton}, 
the mixing time and the cutoff phenomenon were examined 
for non-backtracking random walks. 

Here we will need an estimate for the norm of the $k$-th power of
the non-backtracking operator in the special case of $G=T_d$.
It is known that the spectral radius of $B$ is $\sqrt{d-1}$ in this case.
(See \cite{nb_spectrum} for results on the non-backtracking spectrum 
for the universal cover of any finite graph.) 
It immediately follows that
$$ \| B^k \| = \left( \sqrt{d-1} + o(1) \right)^k ,$$
which implies a bound $ \left( \frac{ 1+o(1) }{ \sqrt{d-1} } \right)^k $
for the correlation in Theorem \ref{thm:edge_corr}.
We want to prove, however, a more explicit bound
in order to get a good quantitative result in Theorem \ref{thm:main}.
To this end we need to more carefully estimate the norm of $B^k$.

\begin{theorem} \label{thm:Bnorm}
Let $B$ be the non-backtracking operator of $T_d$.
Then for any positive integer $k$ we have
$$ \| B^k \| \leq (k+1) (\sqrt{d-1})^{k+1} .$$
\end{theorem}
We postpone the proof until Section \ref{sec:Bnorm}.

\subsection{Decomposition of the Koopman representation}
In this section we briefly explain how the $\Ga$-action on $L^2(M^S, \mu^S)$
can be decomposed into the sum of quasi-regular representations of $\Ga$.
(This can be found in \cite[Section 3]{kechris_tsankov}.
See also \cite[Theorem 2.1, Corollary 2.2]{lyons_nazarov} and
\cite[Section 3.2]{V} for the special case when $S$ is the vertex set of a Cayley graph.)
This will help us to understand the behavior of the
non-backtracking operator on $L^2(M^S, \mu^S)$ in the next section.

A group $\Ga$ acts naturally on the left cosets of
a subgroup $\Delta \leq \Ga$. The corresponding
$\Ga$-action on $\ell^2( \Ga / \Delta)$ is
called \emph{quasi-regular representation}.
If $\Delta$ is the trivial subgroup,
we get the \emph{regular representation} of $\Ga$ on $\ell^2(\Ga)$.
It is easy to see that if $\Delta$ is finite
(which will always be the case in our setting),
then the quasi-regular representation is a subrepresentation of
the regular representation, see \cite[Lemma 3.3]{kechris_tsankov}.

Now let $M$ be a measurable space and $\mu$ a probability measure on $M$.
Suppose that the Hilbert space $L^2(M, \mu)$
has a countable orthonormal basis: $g_0, g_1, g_2, \ldots$,
where $g_0$ will be assumed to be the constant $1$ function.
(The same would work for atomic measures $\mu$
but with a finite orthonormal basis.)
For a countable set $S$ let $\nu$ be the product measure $\mu^S$ on $M^S$.
First we construct an orthonormal basis for $L^2(M^S, \nu)$.
By $\II$ we denote the set of finitely supported $S \to \{0,1,2,\ldots\}$ functions.
For each $q \in \II$ we define an $M^S \to \IR$ function:
$$ W_q(\omega) \defeq \prod_{s \in S} g_{q(s)} \left( \omega_s \right)
\mbox{ for any } \omega = \left( \omega_s \right)_{s \in S} .$$
Note that this is actually a finite product,
since all but finitely many terms are equal to $g_0 \equiv 1$.
According to \cite[Lemma 3.1]{kechris_tsankov}
the functions $W_q$, $q\in \II$ form an orthonormal basis of $L^2(M^S, \nu)$.

Suppose that a countable group $\Ga$ acts on $S$.
Recall that \eqref{eq:action} defines a $\Ga$-action on $M^S$,
which, in turn, induces a $\Ga$-action on $L^2(M^S, \nu)$:
$$ \left( \ga \cdot f \right) (\omega) \defeq
f\left( \ga^{-1} \cdot \omega \right),
\quad f \in L^2(M^S, \nu), \omega \in M^S, \ga \in \Ga .$$
(In representation theory the $\Ga$-action on
$L^2(M^S, \nu)$ is called the \emph{Koopman representation}.)

One can define a $\Ga$-action on $\II$ along the same lines:
$$ \left( \ga \cdot q \right)(s) \defeq q \left( \ga^{-1} \cdot s \right),
\quad q \in \II, s \in S, \ga \in \Ga .$$
The $\Ga$-actions on $\II$ and $L^2(M^S, \nu)$ are compatible in the sense that
\begin{equation} \label{eq:comp_actions}
W_{ \ga \cdot q } = \ga \cdot W_q \mbox{ for any } \ga \in \Ga \mbox{ and } q \in \II .
\end{equation}
It follows that the Koopman representation is equivalent to
the $\Ga$-action on $\ell^2( \II )$.
%There is a general fact that if $\Ga$ acts on a set $\II$,
%then the $\Ga$-action on $\ell^2(\II)$ is equivalent to
%the (direct) sum of quasi-regular representations of $\Ga$.
Given any $q \in \II$, the subspace in $\ell^2(\II)$
corresponding to the orbit $\Ga \cdot q = \{ \ga \cdot q \, : \, \ga \in \Ga\}$
will be invariant under the $\Ga$-action,
and the restriction of the action to this subspace is
equivalent to the quasi-regular representation on $\ell^2(\Ga / \Ga_q)$,
where $\Ga_q$ denotes the stabilizer of $q$.
It follows that the $\Ga$-action on $\ell^2(\II)$ is equivalent to
a direct sum of quasi-regular representations \cite[Proposition 3.2]{kechris_tsankov}.

\subsection{Generalized non-backtracking operator}

We will use the above observations for the case when $\Ga$
somehow corresponds to the directed edge set $E(T_d)$ of the $d$-regular tree.
We claim that there exists a subgroup $\Ga \leq \Aut(T_d)$
that acts \emph{sharply transitively} on the directed edges:
for any pair $e_1, e_2$ of directed edges,
$\Ga$ has a unique element $\ga$ taking $e_1$ to $e_2$.
For a simple proof, draw $T_d$ in the plane and consider all graph
automorphisms $\Phi$ of $T_d$ that preserve the order and the orientation of neighbors
(that is, if $v \in V(T_d)$ has neighbors $v_1, \ldots, v_d$ in a clockwise order,
then $\Phi(v_1), \ldots, \Phi(v_d)$ should be
the neighbors of $\Phi(v)$ also in a clockwise order).
Such automorphisms clearly form a subgroup of $\Aut(T_d)$.
It is also easy to see that prescribing
the image of a directed edge uniquely determines $\Phi$.

In this section $\Ga$ will be a fixed subgroup with the above properties.
There is a one-to-one correspondence between $\Ga$ and $E(T_d)$:
pick a fixed distinguished directed edge $\bar{e}$, and for any $e \in E(T_d)$
let $\ga_e \in \Ga$ be the unique element that takes $e$ to $\bar{e}$. 
The next claim describes the group elements corresponding 
to the predecessors and successors of $e$. 
\begin{claim} \label{cl}
Let $\bar{e}_1, \ldots, \bar{e}_{d-1}$ denote the $d-1$ successors of $\bar{e}$,
that is, $\bar{e} \to \bar{e}_i$, $i=1, \ldots, d-1$.
Then for any directed edge $e$:
\begin{enumerate}[(a)]
\item $\displaystyle \ga_{\bar{e}_i} \ga_e $, $i=1,\ldots,d-1$ correspond to the successors of $e$;
\item $\displaystyle \ga_{\bar{e}_i}^{-1} \ga_e$, $i=1,\ldots,d-1$ correspond to the predecessors of $e$.
%forditott sorrendben is irhatnank a gamma-kat, ugy is igaz lenne
\end{enumerate}
\end{claim}
\begin{proof}
Since $\ga_e$ is a graph automorphism taking $e$ to $\bar{e}$, 
it follows that $\ga_e$ takes the successors of $e$ to the successors 
$\bar{e}_1, \ldots, \bar{e}_{d-1}$ of $\bar{e}$, and by definition 
$\ga_{\bar{e}_i}$ takes $\bar{e_i}$ to $\bar{e}$, see Figure \ref{fig5}. 
This proves (a). 
\begin{figure}[ht]
\centering
\def\svgwidth{0.5\columnwidth}
\begingroup%
  \makeatletter%
  \providecommand\color[2][]{%
    \errmessage{(Inkscape) Color is used for the text in Inkscape, but the package 'color.sty' is not loaded}%
    \renewcommand\color[2][]{}%
  }%
  \providecommand\transparent[1]{%
    \errmessage{(Inkscape) Transparency is used (non-zero) for the text in Inkscape, but the package 'transparent.sty' is not loaded}%
    \renewcommand\transparent[1]{}%
  }%
  \providecommand\rotatebox[2]{#2}%
  \ifx\svgwidth\undefined%
    \setlength{\unitlength}{268.8bp}%
    \ifx\svgscale\undefined%
      \relax%
    \else%
      \setlength{\unitlength}{\unitlength * \real{\svgscale}}%
    \fi%
  \else%
    \setlength{\unitlength}{\svgwidth}%
  \fi%
  \global\let\svgwidth\undefined%
  \global\let\svgscale\undefined%
  \makeatother%
  \begin{picture}(1,0.82768819)%
    \put(0,0){\includegraphics[width=\unitlength,page=1]{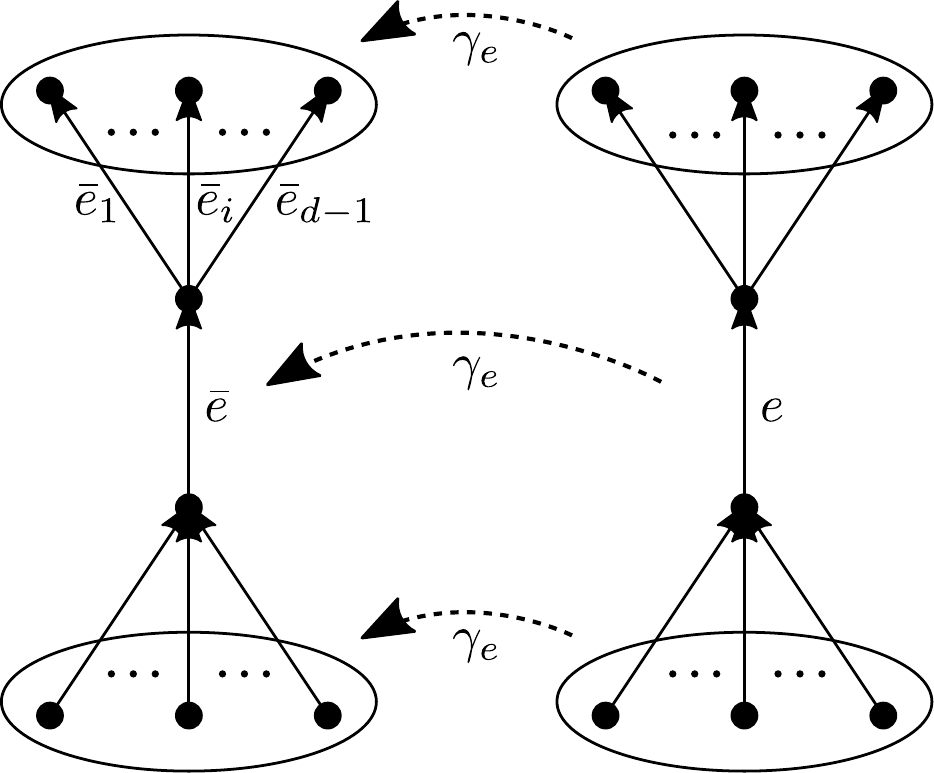}}%
  \end{picture}%
\endgroup%
\caption{Successors and predecessors of $e$}
\label{fig5}
\end{figure}
Similarly, $\ga_e$ takes the predecessors of $e$ to the predecessors of $\bar{e}$ 
(in some order), and it is easy to see that for any predecessor of $\bar{e}$ 
there is a unique $i$ for which $\displaystyle \ga_{\bar{e}_i}^{-1}$ takes 
that predecessor to $\bar{e}$, which clearly proves (b).
\end{proof}
Using the automorphisms $\ga_{\bar{e}_i}$
we can define some kind of a non-backtracking operator $\IB = \IB_\Omega$
on $L^2(\Omega, \nu)$ for any $\Omega$ with a $\Ga$-action:
\begin{equation} \label{eq:gen_nbto}
\IB f \defeq \sum_{i=1}^{d-1} \ga_{\bar{e}_i} \cdot f .
\end{equation}

If $\Omega$ is $\Ga$ itself (with the natural action and $\nu$ being the counting measure),
then $L^2(\Omega, \nu) = \ell^2(\Ga)$, and using (b) in the above claim we get
$$ (\IB f)(\ga_e) = \sum_{i=1}^{d-1} f( \ga_{\bar{e}_i}^{-1} \ga_e) =
\sum_{e' \to e} f(\ga_{e'}) ,$$
which means that in this case $\IB$ is unitarily equivalent to
the ordinary non-backtracking operator $B$ on $\ell^2( E(T_d) )$,
recall \eqref{eq:nbt}.

As for the case when $\Omega$ is $\Ga / \Delta$ for some finite subgroup $\Delta \leq \Ga$,
we saw in the previous section that the $\Ga$-action on $\ell^2( \Ga / \Delta )$
will be a subrepresentation of the regular representation on $\ell^2(\Ga)$,
therefore $\IB$ will be unitarily equivalent to $\left. B \right|_{H}$,
where $B$ is the non-backtracking operator of $T_d$ and
$H$ is some invariant subspace of $\ell^2( E(T_d) )$.

Now let $\Omega$ be $M^S$ with $\nu$ being some product measure $\mu^S$
for a countable set $S$ with a $\Ga$-action.
The only additional assumption we will need is that
the stabilizer of any $s \in S$ is finite.

Let $\II$ be as in the previous section.
We have seen that there is an isometry between
$L^2(M^S, \nu)$ and $\ell^2(\II)$ that preserves the $\Ga$-action.
For $0 \equiv q_0 \in \II$, the orbit of $q_0$ has only one element
(hence the stabilizer $\Ga_{q_0}$ is the whole group $\Ga$).
The corresponding invariant subspace in $L^2(M^S, \nu)$ is
the space of constant functions. Let us focus on the
restriction $\IB_0$ of the non-backtracking operator $\IB$
to the orthogonal complement of the constant functions:
$$ L_0^2(M^S, \nu) \defeq
\left\{ f \in L^2(M^S, \nu) \, : \, \int f \, \mathrm{d} \nu = 0 \right\} .$$
For any other $q_0 \neq q \in \II$ the stabilizer $\Ga_q$ is finite.
It follows from the previous discussion that
$\IB_0$ is unitarily equivalent to the direct sum of
restrictions of $B$ to various invariant subspaces.
Therefore we have the same bound for $\| \IB_0^k \|$ as
we had for $\| B^k \|$ in Theorem \ref{thm:Bnorm}.

We will also need the following formula for $\IB^k$ 
which is an immediate consequence of part (a) of Claim \ref{cl} and
the definition \eqref{eq:gen_nbto} of $\IB$:
\begin{equation} \label{eq:Bk}
\IB^k f = \sum_{\bar{e} \to_k e} \ga_{e} \cdot f .
\end{equation}

\subsection{Edge correlations}

Now we turn to the proof of Theorem \ref{thm:edge_corr}.
Let $\Ga$ still denote a subgroup of $\Aut(T_d)$
that acts sharply transitively on the directed edge set $E(T_d)$.
Set $M = [0,1]$ with $\mu$ being the Lebesgue measure, and $S = V(T_d)$.
(In fact, we can work with arbitrary $M, \mu$,
and any $S$ with an $\Aut(T_d)$-action on it
such that the stabilizer of any element intersected with $\Ga$ is finite.)
%However, from the point of view of factor of i.i.d.\ processes on $\IR^{E(T_d)}$,
%it would not mean greater generality[gb:helyette esetleg: define a broader class] 
%so we might restrict our attention to this special case.)

Let $F \colon M^S \to \IR^{E(T_d)}$ be an $\Aut(T_d)$-factor.
Then for an i.i.d.\ process $Z$ on $M^S$,
$F(Z) = Y = \left( Y_e \right)_{e \in E(T_d)}$ will be
a factor of i.i.d.\ process on $\IR^{E(T_d)}$.

Let $\pi_e \colon \IR^{E(T_d)} \to \IR$ denote
the coordinate projection corresponding to $e \in E(T_d)$. 
Recall that $\bar{e}$ is a fixed distinguished directed edge 
and that $\ga_e \in \Ga$ takes $e$ to $\bar{e}$. 
As we saw in Section \ref{sec:2}, the rule $f \defeq \pi_{\bar{e}} \circ F$
determines $F$. Straightforward calculation shows that
$$ \ga_e \cdot f = \pi_e \circ F \mbox{ implying that }
Y_e = (\ga_e \cdot f)(Z) .$$
Combining this with \eqref{eq:Bk} we get that
$$ \left( \IB^k f \right) (Z) = \sum_{\bar{e} \to_k e} Y_e .$$

The assumption in Theorem \ref{thm:edge_corr} that the variance of $Y_e$ is finite
is equivalent to $f$ being in $L^2(M^s, \nu)$.
Then $\var Y_e = \| f \|_2^2$ for each $e \in E(T_d)$.
Also, since adding a constant does not change the correlation,
we might assume that $\E Y_e = 0$, which means that $f \in L_0^2(M^s, \nu)$. Then
$$ \left\langle f, \IB^k f \right\rangle = \E f(Z) (\IB^k f)(Z) =
\E Y_{\bar{e}} \left( \sum_{\bar{e} \to_k e} Y_e \right) =
\sum_{\bar{e} \to_k e} \E Y_{\bar{e}} Y_e .$$
Since $Y$ is $\Aut(T_d)$-invariant,
$\E Y_{e_1} Y_{e_2}$ is the same for any edges $e_1 \to_k e_2$
(that is, any edges $e_1, e_2$ of distance $k$ and
``pointing to the same direction'').
Therefore the sum on the right-hand side is actually equal to
$(d-1)^k \E Y_{e_1} Y_{e_2} = (d-1)^k \cov( Y_{e_1}, Y_{e_2} )$. 
Using this and the fact that $\| \IB_0^k \| \leq \| B^k \| $ (see the previous section), 
as well as the bound for $\| B^k \|$ in Theorem \ref{thm:Bnorm} we obtain that 
\begin{multline*}
(d-1)^k \left| \cov\left( Y_{e_1}, Y_{e_2} \right) \right| = 
\left| \left\langle f, \IB^k f \right\rangle \right| \leq
\| f \|_2 \cdot \| \IB^k f \|_2 \leq \| \IB^k_0 \|_2 \cdot \| f \|_2^2 \\
\leq (k+1) (\sqrt{d-1})^{k+1} \sqrt{\var(Y_{e_1}) \var(Y_{e_2}) } .
\end{multline*}
Therefore Theorem \ref{thm:edge_corr} follows for the case 
when $e_1$ and $e_2$ point to the same direction.
As for the case when they point away from or towards each other,
we need to do the same for the scalar product
$\left\langle \ga \cdot f, \IB^k f \right\rangle$ and
$\left\langle f, \IB^k ( \ga \cdot f) \right\rangle$, respectively,
where $\ga \in \Ga$ is the unique element that flips $\bar{e}$
(that is, takes $\bar{e}$ to its inverse). 
Since $f \mapsto \ga \cdot f$ is a unitary operator for any fixed $\ga$,
we get the same bound for these scalar products as well.

\subsection{The norms of the powers of the non-backtracking operator}
\label{sec:Bnorm}

In this section we give a proof for Theorem \ref{thm:Bnorm}.
We will follow the arguments presented in \cite[Theorem 4.2]{nb_spectrum}
where they bounded the spectral radius of
the non-backtracking operator on an arbitrary tree without leaves. 

The norm $\|B^k\|$ can be computed as
$$\|B^k\| = \sup_{\|f\| = \|g\| = 1} \langle B^k f, g \rangle,$$
where $\| . \|$ and $\langle . , . \rangle$ denote the standard norm and inner product on $\ell^2(E(G))$. Let us expand this inner product using the definition of $B$.
$$\langle B^k f, g \rangle = \sum_e (B^k f)(e) g(e) =
\sum_{e'{\rightarrow_k} e} f(e')g(e) \le \sum_{e'{\rightarrow_k} e}
|f(e')g(e)|$$
$$\le \sum_{e'{\rightarrow_k} e}
\frac{1}{2}\left(\frac{1}{\alpha_{e,e'}}f^2(e') + \alpha_{e,e'}g^2(e)\right).$$
The last bound is based on the inequality of arithmetic and geometric
means. Note that there is freedom in choosing the positive constants
$\alpha_{e,e'}$ individually for every pair $e'{\rightarrow_k}e$,
which we will discuss later. We
may collect the terms $f^2(e')$ and $g^2(e)$ together:
\begin{equation}
\label{Bnorm-alphabound}
\begin{aligned}
\langle B^k f, g \rangle &\le \frac{1}{2}\sum_{e'} \left(
  \sum_{e'{\rightarrow_k} e} \frac{1}{\alpha_{e,e'}}\right)f^2(e') +  \frac{1}{2}\sum_{e} \left(
  \sum_{e'{\rightarrow_k} e} \alpha_{e,e'}\right) g^2(e)\\
&\le \frac{1}{2} \sup_{e'} \left(
  \sum_{e'{\rightarrow_k} e} \frac{1}{\alpha_{e,e'}}\right) \|f\|^2 + \frac{1}{2} \sup_{e} \left(
  \sum_{e'{\rightarrow_k} e} \alpha_{e,e'}\right) \|g\|^2.
\end{aligned}
\end{equation}

Our goal now is to choose $\alpha_{e,e'}$ in a way 
that the above suprema are as small as possible. 
To this end we fix a root $o$ in $T_d$ and
define the ``spheres'' around this root as
$$V_i \defeq \left\{ v\in V ~|~ d(v,o) = i \right\}.$$
Similarly, we can partition the edges based on their distance from the root:
$$H_i \defeq \left\{ e = (u,v) ~|~ u\in V_{i-1},~v\in V_i \right\} \cup
\left\{ e = (u,v) ~|~ u\in V_{i},~v\in V_{i-1} \right\}.$$
This allows us to define the height of an edge as
$$h(e) \defeq i,~\textrm{if}~e \in H_i.$$
Finally, we set 
$$\alpha_{e,e'} \defeq \left( \frac{1}{ \sqrt{d-1} } \right)^{h(e')-h(e)}.$$
Next we compute what bounds we get for the suprema above.

\begin{claim}
\label{cl:alpha_bound}
For any $e'\in E(G)$ we have
$$\sum_{e'{\rightarrow_k} e} \frac{1}{\alpha_{e,e'}} < (k+1) \sqrt{d-1}^{k+1}$$
\end{claim}
\begin{proof}
We need to browse through all the different configurations of $e$ and
$e'$.

If $e'$ is directed away from the root, then all the edges $e$ reached in
$k$ steps will be $k$ levels above $e'$. Indeed, there is no
possibility to turn back as these are non-backtracking paths. 
This way we reach $(d-1)^k$ different edges, and we get
$$\sum_{e'{\rightarrow_k} e} \frac{1}{\alpha_{e,e'}} = (d-1)^k
\left(\frac{1}{\sqrt{d-1}}\right)^k = \sqrt{d-1}^k.$$

Let us now check an edge $e'$ which is pointing towards the root. 
First we assume that $h(e')>k$. The edges $e$ reached can be
either found going all the way down, or taking $l=0,1,\ldots,k-1$
steps down and then turning up for another $k-l$ steps.

When going simply downwards, there is one such $e$ to reach, with a
height decrease of $k$, thus contributing to the sum by
$$\sqrt{d-1}^k$$
When turning back after $l$ steps, we reach $(d-2)(d-1)^{k-l-1}$
edges, and the height increase is $-l + (k-l-1) = k-1-2l$. The overall
contribution for this $l$ is
$$(d-2)(d-1)^{k-l-1} \left(\frac{1}{\sqrt{d-1}}\right)^{k-1-2l} =
(d-2)\sqrt{d-1}^{k-1} < \sqrt{d-1}^{k+1}.$$
This bound is valid for all $l$ separately, therefore combining
the contributions of all the cases (going all the way down or
turning back after $l=0,1,\ldots,k-1$ steps) we get
$$\sum_{e'{\rightarrow_k} e} \frac{1}{\alpha_{e,e'}} <  (k+1) \sqrt{d-1}^{k+1}.$$

The only case remaining is when $e'$ is pointing towards the root, but
$h(e')\le k$. It is easy to verify that the above method works again,
but some values of $l$ are excluded, and once the intermediate $(d-2)$ factor
increases to $(d-1)$. Nevertheless, the same final bound holds.

We checked all the cases for $e,e'$ and confirmed the stated bound for
every possibility.
\end{proof}

\begin{claim}
\label{cl:alpha_bound2}
For any $e'\in E(G)$ we have
$$\sum_{e'{\rightarrow_k} e} {\alpha_{e,e'}} < (k+1) \sqrt{d-1}^{k+1}$$
\end{claim}
\begin{proof}
Basically the same proof works as in Claim \ref{cl:alpha_bound}, 
only a small adjustment needs to be made due to the change of orientations. 
\end{proof}

Plugging the bounds from Claim \ref{cl:alpha_bound} and \ref{cl:alpha_bound2}
into \eqref{Bnorm-alphabound} we get
$$\langle B^k f, g \rangle \le \frac{1}{2}(k+1)
\sqrt{d-1}^{k+1}\|f\|^2 + \frac{1}{2}(k+1)
\sqrt{d-1}^{k+1}\|g\|^2$$
$$ = (k+1)\sqrt{d-1}^{k+1} ,$$
and this is exactly the bound we were aiming for.

%%%%%%%%%%%%%%%%%%%%%%%%%%%%%%%%%%%%%%%%%%%%%%%%%%%%%%%%%%%%%%%%%%%%%%%%%
%%%%%%%%%%%%%%%%%%%%%%%%%%%%%%%%%%%%%%%%%%%%%%%%%%%%%%%%%%%%%%%%%%%%%%%%%
%%%%%%%%%%%%%%%%%%%%%%%%%%%%%%%%%%%%%%%%%%%%%%%%%%%%%%%%%%%%%%%%%%%%%%%%%

\bibliographystyle{plain}
\bibliography{refs}

\begin{thebibliography}{10}

\bibitem{alon}
Noga Alon, Itai Benjamini, Eyal Lubetzky, and Sasha Sodin.
\newblock Non-backtracking random walks mix faster.
\newblock {\em Commun. Contemp. Math.}, 9(4):585--603, 2007.

\bibitem{nb_spectrum}
Omer Angel, Joel Friedman, and Shlomo Hoory.
\newblock The non-backtracking spectrum of the universal cover of a graph.
\newblock {\em Trans. Amer. Math. Soc.}, 367(6):4287--4318, 2015.

\bibitem{invtree}
\'Agnes Backhausz and Bal\'azs Szegedy.
\newblock On large girth regular graphs and random processes on trees.
\newblock Preprint. arXiv:1406.4420 [math.PR], 2014.

\bibitem{cordec}
{\'A}gnes Backhausz, Bal{\'a}zs Szegedy, and B{\'a}lint Vir{\'a}g.
\newblock Ramanujan graphings and correlation decay in local algorithms.
\newblock {\em Random Structures Algorithms}, 47(3):424--435, 2015.

\bibitem{spec}
{\'A}gnes Backhausz and B\'alint Vir{\'a}g.
\newblock Spectral measures of factor of i.i.d. processes on vertex-transitive
  graphs.
\newblock {\em To appear in Ann. Inst. Henri Poincar\'e Probab. Stat.},
  arXiv:1505.07412 [math.PR], 2015.

\bibitem{karen_ball}
Karen Ball.
\newblock Factors of independent and identically distributed processes with
  non-amenable group actions.
\newblock {\em Ergodic Theory Dyn. Syst.}, 25(3):711--730, 2005.

\bibitem{anna}
Anna Ben-Hamou and Justin Salez.
\newblock Cutoff for non-backtracking random walks on sparse random graphs.
\newblock Preprint. arXiv:1504.02429 [math.PR], 2015.

\bibitem{charles}
Charles Bordenave.
\newblock A new proof of {Friedman's} second eigenvalue theorem and its
  extension to random lifts.
\newblock Preprint. arXiv:1502.04482 [math.CO], 2015.

\bibitem{bowen}
Lewis Bowen.
\newblock The ergodic theory of free group actions: entropy and the
  {$f$}-invariant.
\newblock {\em Groups Geom. Dyn.}, 4(3):419--432, 2010.

\bibitem{sofic}
Lewis Bowen.
\newblock Sofic entropy and amenable groups.
\newblock {\em Ergodic Theory Dynam. Systems}, 32(2):427--466, 2012.

\bibitem{robin}
Clinton~T. Conley, Andrew~S. Marks, and Robin Tucker-Drob.
\newblock Brooks's theorem for measurable colorings.
\newblock Preprint. arXiv:1601.03361 [math.CO], 2016.

\bibitem{endreuj}
Endre Cs\'oka.
\newblock Independent sets and cuts in large-girth regular graphs.
\newblock Preprint. arXiv:1602.02747 [math.CO], 2016.

\bibitem{E}
Endre Cs{\'o}ka, Bal{\'a}zs Gerencs{\'e}r, Viktor Harangi, and B{\'a}lint
  Vir{\'a}g.
\newblock Invariant {G}aussian processes and independent sets on regular graphs
  of large girth.
\newblock {\em Random Structures Algorithms}, 47(2):284--303, 2015.

\bibitem{csokalipp}
Endre Cs\'oka and G\'abor Lippner.
\newblock Invariant random matchings in cayley graphs.
\newblock Preprint. arXiv:1211.2374 [math.CO], 2012.

\bibitem{friedman}
Joel Friedman.
\newblock A proof of {A}lon's second eigenvalue conjecture and related
  problems.
\newblock {\em Mem. Amer. Math. Soc.}, 195(910):viii+100, 2008.

\bibitem{damien}
Damien Gaboriau and Russell Lyons.
\newblock A measurable-group-theoretic solution to von {N}eumann's problem.
\newblock {\em Invent. Math.}, 177(3):533--540, 2009.

\bibitem{gamarnik}
David Gamarnik and Madhu Sudan.
\newblock Limits of local algorithms over sparse random graphs.
\newblock {\em Proceedings of the 5-th Innovations in Theoretical Computer
  Science conference, ACM Special Interest Group on Algorithms and Computation
  Theory}, 2014.

\bibitem{V}
Viktor Harangi and B{\'a}lint Vir{\'a}g.
\newblock Independence ratio and random eigenvectors in transitive graphs.
\newblock {\em Ann. Probab.}, 43(5):2810--2840, 2015.

\bibitem{hoppen}
Carlos Hoppen and Nicholas Wormald.
\newblock Local algorithms, regular graphs of large girth, and random regular
  graphs.
\newblock Preprint. arXiv:1308.0266 [math.CO], 2013.

\bibitem{kechris_tsankov}
A.~S. Kechris and T.~Tsankov.
\newblock Amenable actions and almost invariant sets.
\newblock {\em Proc. Amer. Math. Soc.}, 136(2):687--697 (electronic), 2008.

\bibitem{kempton}
Mark Kempton.
\newblock Non-backtracking random walks and a weighted {Ihara�s} theorem.
\newblock Preprint. arXiv:1603.05553 [math.PR], 2016.

\bibitem{kerr}
David Kerr and Hanfeng Li.
\newblock Soficity, amenability, and dynamical entropy.
\newblock {\em Amer. J. Math.}, 135(3):721--761, 2013.

\bibitem{kun}
G\'abor Kun.
\newblock Expanders have a {spanning Lipschitz} subgraph with large girth.
\newblock Preprint. arXiv:1303.4982 [math.GR], 2013.

\bibitem{lyons}
Russell Lyons.
\newblock Factors of iid on trees.
\newblock {\em Combin. Probab. Comput., to appear}.

\bibitem{lyons_nazarov}
Russell Lyons and Fedor Nazarov.
\newblock Perfect matchings as {IID} factors on non-amenable groups.
\newblock {\em European J. Combin.}, 32(7):1115--1125, 2011.

\bibitem{ornstein_weiss_1987}
Donald~S Ornstein and Benjamin Weiss.
\newblock Entropy and isomorphism theorems for actions of amenable groups.
\newblock {\em J. Analyse Math}, 48:1--141, 1987.

\bibitem{pemantle}
Robin Pemantle.
\newblock Automorphism invariant measures on trees.
\newblock {\em Ann. Probab.}, 20(3):1549--1566, 1992.

\bibitem{doron}
Doron Puder.
\newblock Expansion of random graphs: new proofs, new results.
\newblock {\em Invent. Math.}, 201(3):845--908, 2015.

\bibitem{mustazee}
Mustazee Rahman.
\newblock Factor of iid percolation on trees.
\newblock Preprint. arXiv:1410.3745 [math.PR], 2014.

\bibitem{mustazeebalint}
Mustazee Rahman and B\'alint Vir\'ag.
\newblock Local algorithms for independent sets are half-optimal.
\newblock Preprint. arXiv:1402.0485 [math.PR], 2014.

\bibitem{rokhlin1961construction}
Vladimir~Abramovich Rokhlin and Yakov~G Sinai.
\newblock Construction and properties of invariant measurable divisions.
\newblock {\em Doklady Akademii Nauk SSSR}, 141(5):1038--1041, 1961.

\bibitem{brandon}
Brandon Seward.
\newblock Weak containment and {Rokhlin} entropy.
\newblock Preprint. arXiv:1602.06680 [math.DS], 2016.

\end{thebibliography}

\end{document}